\documentclass[reqno, 12pt ,a4letter]{amsart}
\usepackage{amsmath,amsxtra,amssymb,latexsym, amscd,amsthm}
\usepackage[pdftex,pdfpagelabels]{hyperref}
\usepackage{graphicx,color}
\usepackage[utf8]{inputenc}
\usepackage[mathscr]{eucal}
\usepackage{mathrsfs}
\usepackage[english]{babel}
\usepackage{varioref}
\usepackage{moreverb,graphicx}
\usepackage{enumerate,array}
\usepackage{multirow}
\usepackage{mathpazo}

\newtheorem {theorem}{Theorem}[section]

\newtheorem {lemma}[theorem]{Lemma}
\newtheorem {claim}{Claim}[section]

\theoremstyle{definition}
\newtheorem {definition}[theorem]{Definition}
\newtheorem {remark}[theorem]{Remark}
\newtheorem {example}[theorem]{Example}
\newtheorem {algorithm}{Algorithm}[section]

\newcommand{\ord}{\operatornamewithlimits{ord}}

\def\ees{{\accent"5E e}\kern-.385em\raise.2ex\hbox{\char'23}\kern-.08em}
\def\EES{{\accent"5E E}\kern-.5em\raise.8ex\hbox{\char'23 }}
\def\ow{o\kern-.42em\raise.82ex\hbox{
\vrule width .12em height .0ex depth .075ex \kern-0.16em \char'56}\kern-.07em}
\def\OW{O\kern-.460em\raise1.36ex\hbox{
\vrule width .13em height .0ex depth .075ex \kern-0.16em \char'56}\kern-.07em}

\title{On the separation \L ojasiewicz exponents of real analytic sets in the real plane}

\author{PHI-D{U}NG HO{A}NG$^\dag$}
\address{$^\dag$ Department of Mathematics - Faculty of Fundamental Sciences,
\newline \indent  Posts and Telecommunications Institute of Technology,
\newline \indent 
Km10 Nguyen Trai Rd., Ha Dong District, Hanoi, Vietnam}
\email{dunghp@ptit.edu.vn}

\author{HONG-DUC NGUYEN$^\ddag$}
\address{$^\ddag$TIMAS, Thang Long University, Nghiem Xuan Yem, Hanoi, Vietnam}
\email{duc.nh@thanglong.edu.vn}

\thanks{The first and second author's research is funded by Vietnam National Foundation for Science and Technology Development (NAFOSTED) under grant number 101.04-2024.09.}



\begin{document}
\maketitle

\begin{abstract}
The main aim of the paper is to give a formula for computing the separation \L ojasiewicz exponents for two real analytic set germs via the Newton--Puiseux expansions of their defining functions. Moreover, we present an effective exponent for the case of two real algebraic sets in terms of their degrees.
\end{abstract}

\section{Introduction}
The \L ojasewicz inequalities impressively appeared in the 50s years of last century to solve the question of L. Schwartz for distribution division \cite{Hormander1958,Lojasiewicz1959, Lojasiewicz1965}. They have many relations and applications to other branches of mathematics, such as research on local singularities of analytic functions \cite{Teissier1977}, the proof of Thom's Gradient Conjecture \cite{Kurdyka2000-2}, study of infinitely dimensional version in partial differential equations \cite{Colding2015}, applications in polynomial and tame optimization \cite{HaPham2017,Hoang2016}, \ldots 

Let $A, B$ be real semi-analytic subsets of an open subset $U \subseteq \mathbb{R}^n$ with $0 \in A \cap B$. Then we have the following inequality \cite{Lojasiewicz1965}:
\begin{itemize}
	\item There exist $C, r > 0$ and $\beta \ge 1$ such that 
	\begin{equation}\label{separation-inequality}
		d(x, A) + d(x, B) \geq Cd(x, A \cap B)^\beta, \text{for all}\ \|x\| \leq r,
	\end{equation}
    where $d(x, X)$ is the Euclidean distance of $x \in \mathbb{R}^n$ to the set $X$ ($d(x, X) = 1$ if $X = \emptyset$).
	\item The infimum of such exponents $\beta$ is called {\em separation \L ojasiewicz exponent} of semi-analytic subsets $A$ and $B$ (at the origin), it is denoted by $\mathscr{L}(A,B)$.
\end{itemize} 
Note that the separation \L ojasiewicz exponent $\mathscr{L}(A, B)$ is a rational number \cite[Corollary 2]{Bochnak1975}. The \L ojasiewicz exponent represents and compares the growths of functions on two sides of the inequalities \cite{Lojasiewicz1965}.  

The \L ojasiewicz exponents are used to study some important problems, such as computing some topological invariants \cite{Nguyen2019,Risler1997,Teissier1977}, investigating the singularity at infinity \cite{Ha1990}, the study of the separation of real algebraic sets \cite{Cygan1999,Kurdyka2014} proof of effective Nullstellensatz \cite{Jelonek2005,Ji1992}, etc. There are many works providing computations and estimations of \L ojasiewicz exponents \cite{Gwozdziewicz1999,HD08,HD10,Acunto2005,Kurdyka2014,Nguyen2019,Pham2012}. 

In this paper, we study the separation \L ojasiewicz exponents of the zeros of the real analytic functions in two variables. More precisely, assuming $f$ and $g$ are real analytic functions in two variables, we give a formula to compute the separation \L ojasiewicz exponent $\mathscr{L}(f^{-1}(0),g^{-1}(0))$ in \eqref{separation-inequality} in terms of the approximations of the Newton--Puiseux roots of $f\cdot g$ (see Definition \ref{approximation} and Theorem \ref{separation-theorem}). If $f$ and $g$ are polynomials, then we give an effective formula for the separation \L ojasiewicz exponent. 


\section{Preliminaries} \label{Section2}

\subsection{Newton polygon relative to an arc}\label{subsection-key-lemmas}\quad\\
In this section, we recall the technique of the Newton polygon relative to an arc or sliding-technique due to Kuo-Parusinski \cite{Kuo2000} (see also \cite{HD10}). This technique plays a key role in this article. Let $f: (\mathbb{K}^2,0) \to (\mathbb{K},0)$ ($\mathbb{K}$ is $\mathbb{C}$ or $\mathbb{R}$) be an analytic function germ. Suppose that $f$ is mini-regular in $x$ of order $m$, i.e. in the Taylor expansion of $f$: $$ f(x,y) = f_m(x,y) + f_{m+1}(x,y) + \dots, $$ we have $f_m(1,0) \ne 0$, where $f_k(x,y)$ is the homogeneous component of degree $k$. 

For a {\em Puiseux series}
\begin{eqnarray*}
	x = \phi(y) = c_1y^{n_1/N} + c_2y^{n_2/N}+ \cdots \in \mathbb{K}\{y^{1/N}\}
\end{eqnarray*}
with $N \le n_1 < n_2 < \cdots $ being positive integers, where $c_1 \ne 0$. We define the order of a Puiseux series $x=\phi(y)$:
$$ \mathrm{ord}\phi = \frac{n_1}{N}. $$
The series $\phi$ is said to be real if all coefficients $c_i$ of $\phi$ are real. Let us define 
$$M(X, Y) := f(X + \phi(Y), Y) := \sum c_{ij}X^iY^{j/N}.$$
For each  $c_{ij} \ne 0,$ let us plot a dot at $(i, j/N)$ in $\Bbb R^2$, and call it a {\em Newton dot}. The set of Newton dots is called the {\em Newton diagram}, They generate a convex hull, whose boundary is called the {\em Newton polygon of $f$ relative to $\phi,$}  to be denoted by $\mathbb{P}(f, \phi).$ Note that this is the Newton polygon of $M$ in the usual sense.

\begin{algorithm}[The sliding]\quad
\begin{enumerate}
    \item[\textbf{Input.}]
An analytic function germ $f:(\mathbb{K}^2,0)\to(\mathbb{K},0)$, mini-regular in $x$, and an initial Puiseux series $\phi_0(y)$. Set $k:=0$ and $\phi_0:=\phi$.

	\item[\textbf{Step 1.}] Define $$M_k(X,Y):=f(X+\phi_k(Y),Y) =\sum c^{(k)}_{ij}X^iY^{j/N}$$ and consider the Newton polygon $\mathbb{P}(f,\phi_k)$ of $M_k$. Let $(0,h_k)$ be the lowest Newton dot on $X=0$ and $E_k$ be the compact edge of $\mathbb{P}(f,\phi_k)$ having $(0,h_k)$ as a vertex. Then $$\ord f(\phi_k(y),y)=h_k.$$	
	\item[\textbf{Step 2.}]
	\begin{itemize}
	    \item If $\mathbb{P}(f,\phi_k)$ has no Newton dot on $X=0$, then $f(\phi_k(y),y)=0$ and the algorithm stops with output $$\phi_\infty(y):=\phi_k(y).$$
        \item If $\mathbb{P}(f,\phi_k)$ has a Newton dot on $X=0$, choose a nonzero root $c_k$ of the edge polynomial
	\[
	\mathcal{E}_{E_k}(z)
	:=\sum_{(i,j/N)\in E_k} c^{(k)}_{ij} z^i
	\] and set
	\[
	\phi_{k+1}(y)
	:=\phi_k(y)+c_k\,y^{\tan\theta_{E_k}},
	\]
	where $\theta_{E_k}$ is the angle of the edge $E_k$.
	\end{itemize}
	\item[\textbf{Step 3.}]
	Increase $k:=k+1$ and return to \textbf{Step 1}.
    \item[\textbf{Output.}] The sequence $\{\phi_k(y)\}_k$ has diagram $\phi_0 \rightarrow \phi_1 \rightarrow \cdots\rightarrow \phi_k \rightarrow \cdots$ which produces a Puiseux series $\phi_\infty$ satisfying
\[
f(\phi_\infty(y),y)=0.
\]
\end{enumerate}
\end{algorithm}
The series $\phi_\infty$ is called a {\it final result of the sliding of $\phi$ along $f$}. 
The above algorithm is based on the following lemma.
\begin{lemma} \label{lemma-sliding}
	Suppose that $\phi$ is not a root of $f = 0$. Consider a series 
	$$\psi : x = \phi(y) + c y^{\rho} + o(\rho),$$
	where $c \in \mathbb{K},0<\rho \in \mathbb{Q}$ and $o(\rho)$ denotes a Puiseux series of order greater than $\rho$. Then the following statements hold:
	\begin{itemize}
		\item[(i)] If $\tan \theta_{E_1} < \rho$ or $\tan \theta_{E_1}= \rho$ and $\mathcal{E}_{E_1}(c) \ne 0$ then 
		$\mathbb{P}(f, \phi) = \mathbb{P}(f, \psi),$ and, therefore
		$\mathrm{ord} f(\phi(y), y) = \mathrm{ord} f(\psi(y), y).$
		
		\item[(ii)] If $\tan \theta_{E_1}= \rho$ and $\mathcal{E}_{E_1}(c)  = 0$ then $\mathrm{ord} f(\phi(y), y) < \mathrm{ord} f(\psi(y), y)$.
	\end{itemize}
\end{lemma}
\begin{proof} For a detailed proof, we refer to \cite{HD08}. In fact, the special case where $\psi : x = \phi(y) + c y^{\tan \theta_{E_1}}$ was proved in \cite[Lemma 2.1]{HD08}. The lemma is then deduced by applying the special case (infinitely) many times.
\end{proof}

We define
$$
\mathrm{ord}d(\phi, V_f) := \begin{cases}\label{order-distance}
	\max\limits_{j} \{\mathrm{ord}(\phi(y)-\beta_j(y))\},\ \text{$\beta_j(y)$ are the real Newton-Puiseux roots of $f$,}\\
	1\ \text{ if $f$ has no real roots,}
\end{cases} 	
$$
where $V_f$ denotes the zero locus $f(x,y)=0$ of $f$.
\begin{lemma}\label{distant-sliding}
Suppose that $\phi$ is not a root of $f = 0$. Then
	\begin{itemize}
		\item[(i)] $\mathrm{ord}d(\phi, V_f)\leq \tan \theta_{E_1}.$		
		\item[(ii)] Let $\psi$ be a series of form
	$$\psi : x = \phi(y) + c y^{\rho} + o(\rho),\ c \in \mathbb{R}.$$
Suppose that $\rho>\mathrm{ord}d(\psi, V_f)$, or $\rho\geq \mathrm{ord}d(\psi, V_f)$ and $c$ is generic. Then
		$$\mathrm{ord}d(\psi, V_f)= \mathrm{ord}d(\phi, V_f).$$
	\end{itemize}
\end{lemma}
\begin{proof}
(i) Assume by contradiction that $\rho:=\mathrm{ord}d(\phi, V_f)> \tan \theta_{E_1}$, then taking a real Newton-Puiseux root $\beta(y)$ such that $\mathrm{ord}d(\phi, V_f)=\mathrm{ord}(\phi(y)-\beta(y))$ we obtain
$$\beta(y)= \phi(y) + c y^{\rho} + o(\rho), c\neq 0.$$
By Lemma \ref{lemma-sliding}, $\mathrm{ord} f(\phi(y), y) = \mathrm{ord} f(\beta(y), y)=+\infty$, which is a contradiction.

(ii) Let $\beta$ be any Newton-Puiseux root of $f$. Then
$$\psi(y)-\beta(y)=\phi(y)-\beta(y)+ c y^{\rho} + o(\rho).$$
Since $\mathrm{ord}\left(\psi(y)-\beta(y)\right)\leq \mathrm{ord}d(\phi, V_f)\leq \rho$, it follows that
$$\mathrm{ord}\left(\psi(y)-\beta(y)\right)=\mathrm{ord}\left(\phi(y)-\beta(y)\right).$$
Therefore
$$\mathrm{ord}d(\psi, V_f)=\mathrm{ord}d(\phi, V_f).$$
\end{proof}
The following lemma is similar to \cite[Proposition 2.3]{HD10}.
\begin{lemma}\label{distance}
	Assume that $\mathrm{ord}\phi(y) \ge 1$. Then
	$$ d( (\phi(y),y), V_f ) \simeq |y|^{\mathrm{ord}d(\phi, V_f)}, \text{ as } y\to 0$$

\end{lemma}
\begin{proof}(cf. \cite[Lemma 2.16]{Kuo1974}). Let us consider the following two cases.
	
	If $f=0$ does not have any real Newton--Puiseux root, then $V_f= (0,0)$. Taking $\epsilon > 0$ small sufficiently such that $|y| < \epsilon$, we have $$d((\phi(y),y),V) = d((\phi(y),y), (0,0)) = \|(\phi(y),y)\| = \sqrt{(\phi(y))^2 + y^2}.$$ Since the assumption $\mathrm{ord}\phi(y) \ge 1$, then 
	\begin{align}\label{Eqn8}
	d((\phi(y),y),V) \simeq |y|^1\ \text{for all}\ |y| < \epsilon.	
	\end{align}
	
	If $f=0$ has some real Newton--Puiseux roots, then there exists $\epsilon > 0$ sufficiently small such that $f(x,y)$ can be represented as follows
	\begin{align*}
		f(x,y) = u(x,y)\prod_{j=1}^{k}(x - \beta_i(y)), 
	\end{align*}
	where $u(x,y)$ has no real roots and $|y| < \epsilon$. Hence
	\begin{align*}
		d( (\phi(y),y), V ) = \inf_j\|(\phi(y),y) - (\beta_j(y), y)\| = \inf_j|\phi(y) - \beta_j(y)|.
	\end{align*}
	It follows that 
	\begin{align*}
		\inf_j|\phi(y) - \beta_j(y)| \simeq |y|^{\max_j\{\mathrm{ord}(\phi(y)-\beta_j(y))\}}.
	\end{align*}
	Therefore, $d((\phi(y),y),V) \simeq |y|^{\mathrm{ord}d(\phi, V)}$  for all $|y| < \epsilon$. Combining this with \eqref{Eqn8}, we deduce the proof of the lemma.
\end{proof}

\subsection{Approximations of Puiseux series}\label{real-approx}
Let $x = \gamma(y)$ be a Puiseux series in the following form:
\begin{eqnarray*}
	\gamma(y) = a_1 y^{n_1/N} + a_2y^{n_2/N}+ \cdots + a_{s - 1} y^{n_{s - 1}/N} + c_s y^{n_s/N} +  \cdots,
\end{eqnarray*}
where $a_i \in \mathbb{R}$ and $c_s$ is the first non-real coefficient, if there is one. Let us replace $c_s$ by a generic real number $g,$ and call
\begin{equation}\label{real-approximation}
	\gamma^{\mathbb{R}}(y) := a_1 y^{n_1/N} + a_2y^{n_2/N}+ \cdots + a_{s - 1} y^{n_{s - 1}/N} + g y^{n_s/N},
\end{equation}
a {\em real approximation of $\gamma$}. In case $s = +\infty,$ let $\gamma^{\mathbb{R}} = \gamma.$

Let $f\colon (\mathbb{R}^2,0) \to (\mathbb{R},0)$ be a real analytic function germ, which is mini-regular in $x$ of order $m$. Suppose that $\phi$ is not a Newton--Puiseux root of $f = 0$ and $\gamma$ is a final result of the sliding of $\phi$ along $f$. Then
$$\gamma(y)  = \phi(y) + c y^{\tan \theta_{E_1}}+ \text{higher order terms},$$
where $\mathcal{E}_{H_{M}}(c) = 0$ and $M(X,Y) = f(X + \phi(Y),Y)$. If $c \notin \mathbb{R}$, then by the definition of the real approximation \eqref{real-approximation}, the real approximation $\gamma^{\mathbb{R}}$ of the series $\gamma$ also has the form 
\begin{equation}\label{real-approximation2}
	\gamma^{\mathbb{R}}(y) = \phi(y) + g y^{\tan \theta_{E_1}},	
\end{equation}	
where $g \in \mathbb{R}$ is generic. For $f \in \mathbb{K}\{x,y\}$ which is regular in $x$, let $\mathcal{V}_{\mathbb{R}}(f)$ be the set of all real approximations of all the Newton-Puiseux roots of $f$.
\begin{definition}[\cite{Nguyen2019}]
Let $\phi$ be a Puiseux series $\phi(y) = \sum a_iy^{\alpha_i}$. For each positive real number $\rho$, the {\em $\rho$-approximation} of $\phi(y)$ is defined by series 
$$ \sum_{\alpha_i < \rho} a_iy^{\alpha_i} + cy^\rho, $$ where $c$ is a generic real number.
\end{definition}
\begin{remark}
Let $\phi(y)$ be a non-real Puiseux series. Then the real approximation $\phi^{\mathbb{R}}(y)$ of $\phi(y)$ is the $\rho$-approximation of $\phi$, where $\rho$ is the smallest exponent occurring in $\phi$ with non-real coefficient. 
\end{remark}
\begin{definition}\label{approximation}
    Let $\phi_1, \phi_2$ be any two distinct Puiseux series. The series $\phi_{1,2}$ is called the {\em approximation of $\phi_1, \phi_2$} if it is the $\rho$-approximation of $\phi_1$ (also the $\rho$-approximation of $\phi_2$), where $\rho := \text{ord}(\phi_1 - \phi_2)$. Let $f\colon (\mathbb{R}^2,0) \to (\mathbb{R},0)$ be a real analytic function germ, which is mini-regular in $x$. Let $x=\phi_i(y), i=1,\ldots,s$ be the set of the Newton-Puiseux roots of $f$. We denote by $\mathcal{V}_a(f)$ the set of all approximations $\phi_{i,j}$ of $\phi_i$ and $\phi_j$ such that $\phi_{i,j}$ is a real Puiseux series. Note that $\mathcal{V}_{\mathbb{R}}(f) \subset \mathcal{V}_a(f)$ if $f \in \mathbb{R}\{x,y\}$.
\end{definition}

\begin{example}
	Let $\phi_1(y) = y^{\frac{3}{2}} - 2y^{\frac{5}{2}} + 3iy^{\frac{7}{2}} + \dots$ and $\phi_2(y) = y^{\frac{3}{2}} - 2y^{\frac{5}{2}} - 3y^{\frac{7}{2}} + iy^{\frac{9}{2}} + \dots$. Then, the approximation of series $\phi_1$ and $\phi_2$ is $\phi_{1,2}(y) = y^{\frac{3}{2}} - 2y^{\frac{5}{2}} + gy^{\frac{7}{2}}$, where $g \in \mathbb{R}$ and generic.
\end{example}
\section{Separation \L ojasewicz exponent of two real analytic subsets}\label{Section3}
Let $f, g \colon (\mathbb{R}^2,0) \to (\mathbb{R},0)$ be reduced real analytic functions germs in two variables, which are mini-regular in $x$ of order $m$. Put $V_f = f^{-1}(0), V_g = g^{-1}(0)$. Then, by the \L ojasiewicz inequality \eqref{separation-inequality}, there exist $C, r > 0$ and $\beta \ge 1$ such that
\begin{equation}\label{separation}
	d(x, V_f) + d(x, V_g) \geq Cd(x, V_f \cap V_g)^\beta \text{ for all}\ \|x\| \leq r.
\end{equation} 


Let $\varphi\colon [0,\varepsilon)\to \Bbb R^2$ be a semi-algebraic analytic curve such that $\varphi(0)=0$. Then there is a rational number $\ell(\varphi)$ such that
$$d(\varphi(t), V_f)+d(\varphi(t), V_g)\sim \left(d(\varphi(t), V_f \cap V_g)\right)^{\ell(\varphi)}\text{ as } t\to 0.$$
Let $\mathscr{L}(V_f, V_g)$ denote the separation \L ojasiewicz exponent of $V_f$ and $V_g$. Using the Curve Selection Lemma (see \cite{Milnor1968}), we can easily prove the following equality
\begin{eqnarray}\label{Eqn10}
	\mathscr{L}(V_f, V_g) &=& \sup_\phi \ell(\phi),
\end{eqnarray}
where the supremum is taken over all semi-algebraic analytic curves passing through the origin, which do not contain in $V_f \cap V_g$.

Suppose that the curve $\varphi$ is parametrized by a Puiseux series $x= \phi(y)$. We can see that
\begin{equation}\label{lphi-separation}
	\ell(\varphi)=\ell(\phi) := \frac{ \min\{\mathrm{ord}d(\phi, V_f), \mathrm{ord}d(\phi, V_g)\} }{\mathrm{ord}d(\phi,V_{\gcd(f,g)})}.	
\end{equation} 
We define the number
\begin{align}\label{L+_formula}
   \mathscr{L}_+(V_f, V_g) := \max\{\ell(\gamma)| \gamma \in \mathcal{V}_a(f \cdot g)\},
\end{align}
where the set $\mathcal{V}_a(f \cdot g)$ is defined in Definition \ref{approximation}. Let us denote by $\overline{f}, \overline{g}$ the real analytic function germs defined by $\overline{f}(x, y) := f(x, -y)$ and $\overline{g}(x, y) := g(x, -y)$, and set $\mathscr{L}_{-}(V_f, V_g) := \mathscr{L}_{+}(V_{\overline{f}}, V_{\overline{g}})$. Now, we establish the formula for the separation \L ojasiewicz exponent $\mathscr{L}(V_f, V_g)$.
\begin{theorem}\label{separation-theorem}
	Let $f, g \colon (\mathbb{R}^2,0) \to (\mathbb{R},0)$ be non-zero reduced real analytic function germs, which are mini-regular in $x$ of order $m$. Then
	$$\mathscr{L}(V_f, V_g) = \max\{\mathscr{L}_{+}(V_f, V_g), \mathscr{L}_{-}(V_f, V_g)\},$$ where $\mathscr{L}_{-}(V_f, V_g) := \mathscr{L}_{+}(V_{\overline{f}}, V_{\overline{g}})$ and $\mathscr{L}_{+}(V_{\overline{f}}, V_{\overline{g}})$ is defined by \eqref{L+_formula}.
\end{theorem}
First, we need some lemmas. Put $h := \gcd(f,g)$, where $\gcd(f,g)$ is the greatest common division of $f$ and $g$. Let us consider the Newton polygons relative to any arc $x=\phi(y)$, they are $\mathbb{P}(f, \phi)$, $\mathbb{P}(g,\phi)$ and $\mathbb{P}(h,\phi)$ with the highest Newton edges $F_1$, $G_1$ and $H_1$, respectively. Assume that $\mathcal{E}_{F_1}$, $\mathcal{E}_{G_1}$ and $\mathcal{E}_{H_1}$ are their associated polynomials. Let $\theta_{F_1}$, $\theta_{G_1}$ and $\theta_{H_1}$ be the Newton angles of $F_1$, $G_1$ and $H_1$, respectively. It is easy to see that 
	\begin{align}\label{compare-3-tan}
		\min\{\tan\theta_{F_1}, \tan\theta_{G_1}\} \geq \tan\theta_{H_1}.
	\end{align}
	\begin{lemma}\label{claim1} 
		Suppose that $\ell(\phi) > \mathscr{L}_+(V_f, V_g)$. Then 
    \begin{itemize}
        \item[(i)] $\tan\theta_{F_1} = \tan\theta_{G_1}.$
        \item[(ii)] The polynomial $\mathcal{E}_{F_1}\mathcal{E}_{G_1}$ has only one root, which is a real value.
    \end{itemize}
		\end{lemma}
	\begin{proof}
    By contradiction we assume that $\tan\theta_{F_1} \neq \tan\theta_{G_1}.$ Without loss of generality suppose that $\tan\theta_{F_1} > \tan\theta_{G_1}.$ Let $\phi_{1,\infty}$ and $\phi_{2,\infty}$ be final results of sliding $\phi$ along $f$ and $g$, respectively. Then we have 
	\begin{align}
		\phi_{1,\infty}(y) &= \phi(y) + \sum_{i \ge 1}a_iy^{\alpha_i}, a_1\neq 0\\
		\phi_{2,\infty}(y) &= \phi(y) + \sum_{i \ge 1}b_iy^{\beta_i}, b_1\neq 0,
	\end{align}
	where $\tan\theta_{F_1}=\alpha_1 <\alpha_2 < \ldots$ and $\tan\theta_{G_1} =\beta_1 < \beta_2 <
    \ldots $.
    
    Let $\phi_{1,2}$ be the approximation of $\phi_{1,\infty}$ and $\phi_{2,\infty}$. Then $\phi_{1,2}\in \mathcal{V}_a(fg)$ and
	$$ \phi_{1,2}(y) = \phi(y) + cy^{\tan\theta_{G_1}} + o(\tan\theta_{G_1}), $$ where $c \in \mathbb{R}$ is generic. It follows from Lemma \ref{distant-sliding} that
    	\begin{align}\label{order-with-Vg}
    		\mathrm{ord}d(\phi_{1,2}, V_g) = \mathrm{ord}d(\phi, V_g) \text{ and } \mathrm{ord}d(\phi_{1,2}, V_h) = \mathrm{ord}d(\phi, V_h).
    	\end{align}
    	 
We will show that $\ell(\phi_{1,2})=\ell(\phi)$ by considering the following two cases.

\begin{itemize}
	\item Case 1: 	$\mathrm{ord}d(\phi, V_f) \leq \tan\theta_{G_1}$. Then $\mathrm{ord}\ d(\phi_{1,2}, V_f) = \mathrm{ord}\ d(\phi, V_f)$ due to Lemma \ref{distant-sliding}. Hence $\ell(\phi_{1,2})=\ell(\phi)$.
    
	\item Case 2: $\mathrm{ord}d(\phi, V_f) > \tan\theta_{G_1}$. Take a real Newton-Puiseux root $\gamma$ of $f$ satisfying
	\begin{align*}
		\mathrm{ord}(\phi - \gamma) = \mathrm{ord}\ d(\phi, V_f).
	\end{align*}
    Since
    $$ \phi_{1,2}(y)- \gamma(y) =\phi(y)-\gamma(y) + cy^{\tan\theta_{G_1}} + o(\tan\theta_{G_1}), $$
     it follows that $\mathrm{ord}d(\phi_{1,2} - \gamma) = \tan\theta_{G_1}$. Therefore,
	\begin{align}\label{order-with-Vf-tan}
		\mathrm{ord}d(\phi_{1,2}, V_f) \ge \tan\theta_{G_1}\ge \mathrm{ord}d(\phi_{1,2}, V_g).
	\end{align}
	Hence,
	  		\begin{align*}
			\ell(\phi_{1,2}) &= \frac{\min\{\mathrm{ord}d(\phi_{1,2},V_f), \mathrm{ord}d(\phi_{1,2},V_g)\}}{\mathrm{ord}d(\phi_{1,2},V_h)} \\
			&= \frac{ \mathrm{ord}d(\phi,V_g) }{\mathrm{ord}d(\phi,V_h)}  =\ell(\phi).
		\end{align*} 
		
\end{itemize}
The equality $\ell(\phi_{1,2})=\ell(\phi)$ contradicts the assumption that $\ell(\phi) > \mathscr{L}_+(V_f, V_g)$.  This proves (i).

(ii) By (i) one has $\rho=\tan \theta_{F_1} = \tan\theta_{G_1}$. Then $\mathcal{E}_{F_1}\mathcal{E}_{G_1}$ is the polynomial associated with the highest Newton edge of $\mathbb{P}(fg, \phi)$ with the angle $\rho$. Suppose by contradiction that $\mathcal{E}_{F_1}\mathcal{E}_{G_1}$ has two distinct roots $c_1,c_2$. Let $\varphi_i=\phi+c_i y^{\rho }$ for each $i=1,2$, and let $\varphi_{i,\infty}$ be a final result of sliding $\varphi_i$ along $fg$. Let $\varphi_{1,2}$ be the approximation of $\varphi_{1,\infty}$ and $\varphi_{2,\infty}$. Then $\varphi_{1,2}\in \mathcal{V}_a(fg)$ and 
	$$ \varphi_{1,2}(y) = \phi(y) + cy^{\rho} + o(\rho), $$ where $c \in \mathbb{R}$ is a generic number. Note that 
    $$\tan \theta_{H_1}\leq \tan \theta_{F_1}=\tan \theta_{G_1}=\rho.$$
    It follows from Lemma \ref{distant-sliding} that
    $$\mathrm{ord}d(\varphi_{1,2},V_f)=\mathrm{ord}d(\phi,V_f),\ \mathrm{ord}d(\varphi_{1,2},V_g)=\mathrm{ord}d(\phi,V_g) $$
	and
    $$d(\varphi_{1,2},V_h)=\mathrm{ord}d(\phi,V_h).$$
    Hence $\ell(\phi_{1,2})=\ell(\phi)$, which is a contradiction. The lemma is proved.    
\end{proof}

\begin{proof}[Proof of Theorem \ref{separation-theorem}]
	By the formula \eqref{Eqn10}, it is obvious that $$ \mathscr{L}(V_f, V_g) \geq \max\{\mathscr{L}_{+}(V_f, V_g), \mathscr{L}_{-}(V_f, V_g)\}. $$
	
	Suppose for contradiction that $$ \mathscr{L}(V_f, V_g) > \max\{\mathscr{L}_{+}(V_f, V_g), \mathscr{L}_{-}(V_f, V_g)\}. $$
	Then there is a real analytic arc $\phi$ passing through the origin and not lying the $x$-axis such that 
	\begin{align}\label{contradicted-assumption}
	 \ell(\phi) > \max\{\mathscr{L}_{+}(V_f, V_g), \mathscr{L}_{-}(V_f, V_g)\}	
	\end{align}
	and the parameter forms of $\phi$ is 
	$$ (x = \phi(t), y = t)\ \text{or}\ (x = \phi(t), y = -t),\ \text{where}\ \phi(t) \in \mathbb{R}\{t^{1/N}\}. $$
    Without loss of generality, assume that $\phi$ can be parametrized by $(x = \phi(t), y = t)$. 
    By Lemma \ref{claim1}, then $\theta_{F_1} = \theta_{G_1}$ and the polynomial $\mathcal{E}_{F_1}\mathcal{E}_{G_1}$ has a unique root $c_0 \in \mathbb{R}$. Define 
$$ \tilde\phi(y) := \phi(y) + c_0y^{\tan\theta_{F_1}}. $$
\begin{claim}\label{claim4} We have
	$$\ell(\tilde\phi) \ge \ell(\phi).$$
\end{claim}
\begin{proof}
We first see that $$\min\{\mathrm{ord}d(\phi, V_f),\  \mathrm{ord}d(\phi, V_g)\} > \mathrm{ord}d(\phi, V_f \cap V_g)=\mathrm{ord}d(\phi, V_h),$$
because $\ell(\phi)>1$. Then $$\tan\theta_{F_1}\geq \mathrm{ord}d(\phi, V_f)> \mathrm{ord}d(\phi, V_h).$$
Applying Lemma \ref{distant-sliding} and its proof, one has
$$\mathrm{ord}d(\tilde\phi, V_h)=\mathrm{ord}d(\phi, V_h),\ \mathrm{ord}d(\tilde\phi, V_f)\geq \mathrm{ord}d(\phi, V_f) \text{ and } \mathrm{ord}d(\tilde\phi, V_g)\geq\mathrm{ord}d(\phi, V_g).$$
Hence $\ell(\tilde\phi) \ge \ell(\phi).$
\end{proof}

Consider Newton polygons $\mathbb{P}(f,\tilde\phi)$ and $\mathbb{P}(g,\tilde\phi)$. The polygon $\mathbb{P}(f,\tilde\phi)$ has the highest Newton edge $\widetilde{F}_1$ with angle $\theta_{\widetilde{F}_1}$ and the associated polynomial $\mathcal{E}_{\widetilde{F}_1}$; the polygon $\mathbb{P}(g,\tilde\phi)$ has the highest Newton edge $\widetilde{G}_1$ with angle $\theta_{\widetilde{G}_1}$ and the associated polynomial $\mathcal{E}_{\widetilde{G}_1}$. Combining Claim \ref{claim4} and Lemma \ref{claim1} we obtain 

\begin{claim}\label{claim3}
	If $\tilde\phi$ is not a Newton--Puiseux root of $h=0$, then
	\begin{enumerate}
		\item $ \tan\theta_{\widetilde{F}_1} = \tan\theta_{\widetilde{G}_1}$;
		\item The polynomial $\mathcal{E}_{\widetilde{F}_1}\mathcal{E}_{\widetilde{G}_1}$ has only real root.
	\end{enumerate}
\end{claim}

Let $\phi_0:=\phi$ and $$\phi_1(y)=\tilde\phi(y) = \phi_0(y) + c_0y^{\tan\theta_{F_1}}.$$ 
Then $\phi_1$ is the unique sliding of $\phi_0$ along $f$ (as well as $g$ and $h$). Applying this process, we obtain a sequence 
\begin{align}
    \phi \to \gamma_1 \to \gamma_2 \to \dots \to \gamma_n \to \dots \to \phi_\infty,
\end{align}
where $\phi_{k+1}$ be the unique sliding of $\phi_k$ along $h$ for each $k\geq 1$. The series $\phi_\infty$ is a real Newton-Puiseux root of $h$. It yields the following identities
$$\mathrm{ord}d(\phi, V_f)=\mathrm{ord}d(\phi, V_g)=\mathrm{ord}d(\phi, V_h)=\mathrm{ord}(\phi-\phi_\infty)=\tan \theta_{F_1}.$$
This implies $\ell(\phi)=1$, which is a contradiction. So, the theorem is proved.
\end{proof}

\section{Effective estimation for separation \L ojasiewicz exponent}\label{section4}
Suppose that $f$ and $g$ are polynomials of two real variables, then there is the following effective version of the separation \L ojasiewicz exponent.
\begin{theorem}
    Let $f$ and $g$ be polynomials of two real variables such that their degrees are not greater than $d$ and $f(0)=g(0)=0$. Then there exist positive numbers $C$ and $r$ such that 
    \begin{equation}\label{effective-separation}
	d(x, V_f) + d(x, V_g) \geq Cd(x, V_f \cap V_g)^{\frac{(2d-1)^2+1}{2}},\ \text{for all}\ \|x\| \leq r.
\end{equation}
\end{theorem}
\begin{proof}(cf. \cite{Kurdyka2014})
Put $F := f^2 + g^2$, then $\deg F \le 2d$, where $d = \max\{\deg f, \deg g\}$. Moreover, $V_F = V_f \cap V_g$. 

If $V_f \subset V_g$, then $\mathscr{L}(V_f, V_g) = 1$. Assume that $V_f \not\subset V_g$. Applying the classical \L ojasiewicz inequality to $F$, there exists $\alpha, c > 0$ such that
$$ |F(x)| \ge  c d(x, V_F)^\alpha, \forall x \in U \cap (V_f \setminus V_g),$$
where $U \subset \mathbb{R}^2$ is a neighbourhood of $(0,0)$. 
Let $\widetilde{\mathcal{L}}(F)$ be the classical \L ojasiewicz exponent of the polynomial $F$, then, by \cite[Theorem 4.4]{Nguyen2019}, $\widetilde{\mathcal{L}}(F) \le (2d-1)^2+1$. Hence, the effective optimal version of classical \L ojasiewicz inequality is 
$$ |F(x)| \ge  c d(x, V_F)^{(2d-1)^2+1}, \forall x \in U \cap (V_f \setminus V_g), $$
Note that $x \in U \cap (V_f \setminus V_g)$, then
\begin{align}\label{F and g}
|F(x)| = |g(x)|^2.    
\end{align}
Therefore,
\begin{align}
   |g(x)| \ge  c d(x, V_f \cap V_g)^{\frac{(2d-1)^2+1}{2}}, \forall x \in U \cap (V_f \setminus V_g),
\end{align}
It is easy to see that there exists $C, C' > 0$ such that
\begin{align}\label{sim of d and g}
   Cd(x, V_g) \ge C'|g(x)|. 
\end{align}
Combining \eqref{F and g} with \eqref{sim of d and g}, we obtain 
\begin{align*}
    d(x, V_g) \ge Cd(x, V_f \cap V_g)^{\frac{(2d-1)^2+1}{2}}, \forall x \in U \cap (V_f \setminus V_g).
\end{align*}
This implies inequality \eqref{effective-separation}.
\end{proof}
\begin{remark}
The exponent in inequality \eqref{effective-separation} is sharper than the exponent in \cite[Corollary 8]{Kurdyka2014} in the case of two variables, where $0 \in \mathbb{R}^2$ is not necessarily isolated point of $V_f \cap V_g$. 
\end{remark}

\bibliographystyle{amsalpha}

\end{document}